\documentclass[11pt]{article}
\usepackage[a4paper]{geometry}
\usepackage[linesnumbered,ruled,vlined]{algorithm2e}
\usepackage{amsfonts,amsmath,amsopn,amssymb}
\usepackage{braket}
\usepackage{enumitem}
\usepackage{graphicx}
\usepackage{epstopdf}
\usepackage{subcaption}
\usepackage{caption}
\usepackage[colorlinks,linkcolor=blue,anchorcolor=blue,citecolor=blue]{hyperref}
\usepackage{cleveref}
\usepackage{ntheorem}
\usepackage{url}
\usepackage{xcolor}
\usepackage{booktabs}
\usepackage{multirow}
\usepackage{diagbox}
\numberwithin{equation}{section}

\theoremheaderfont{\normalfont\bf}
\theorembodyfont{\normalfont\rm}

\crefname{algorithm}{Algorithm}{Algorithms}
\Crefname{algorithm}{Algorithm}{Algorithms}

\newtheorem{definition}{Definition}[section]
\newtheorem{remark}{Remark}[section]
\newtheorem{theorem}{Theorem}[section]

\newtheorem{lemma}{Lemma}[section]

\newtheorem{example}{Example}[section]

\newenvironment{keywords}{{\noindent\it {\bf Key~words.}}\quad}{}
\newenvironment{proof}{{\noindent\it Proof.}\quad}{\hfill $\square$\par}

\DeclareMathOperator{\diag}{diag}

\DeclareMathOperator{\rank}{rank}

\SetKwInput{KwInput}{Input}         
\SetKwInput{KwOutput}{Output}       
\usepackage{fancyhdr} 

\fancypagestyle{titlepage}{ 
    \fancyhf{}
    \fancyhead[C]{
        \small This is the accepted manuscript of the paper, which has been accepted for publication in \textit{Linear and Multilinear Algebra}. The final published version may differ. 
    }
    \fancyfoot{} 
}

\begin{document}
\thispagestyle{titlepage} 

\begin{center}
    {\LARGE LU Decomposition and Generalized Autoone-Takagi Decomposition of Dual Matrices and their Applications} \\[1em]

    {\large Renjie Xu\footnote{Corresponding author (R. Xu). Center for Intelligent Multidimensional Data Analysis, Hong Kong Science Park, Shatin, Hong Kong, China. Email: \href{mailto:renjie@innocimda.com}{renjie@innocimda.com}. This author is supported by the Hong Kong Innovation and Technology Commission (InnoHK Project CIMDA).}}\\[1em]

    {\large Yimin Wei\footnote{School of Mathematical Sciences and Shanghai Key Laboratory of Contemporary Applied Mathematics, Fudan University, Shanghai, 200433, China. Email: \href{mailto:ymwei@fudan.edu.cn}{ymwei@fudan.edu.cn}. This author is supported by the National Natural Science Foundation of China under grant 12271108 and the Ministry of Science and Technology of China under grant G2023132005L.}}\\[1em]

    {\large Hong Yan\footnote{Department of Electrical Engineering and Center for Intelligent Multidimensional Data Analysis, City University of Hong Kong, Hong Kong, China. Email: \href{mailto:h.yan@cityu.edu.hk}{h.yan@cityu.edu.hk}. This author is supported by the Hong Kong Research Grants Council (Project 11204821), the Hong Kong Innovation and Technology Commission (InnoHK Project CIMDA), and the City University of Hong Kong (Projects 9610034 and 9610460).}}\\[2em]

    {\large \today} 
\end{center}

\begin{abstract}
This paper uses matrix transformations to provide the Autoone-Takagi decomposition of dual complex symmetric matrices and extends it to dual quaternion $\eta$-Hermitian matrices. 
The LU decomposition of dual matrices is given using the general solution of the Sylvester equation, and its equivalence to the existence of rank-k decomposition and dual Moore-Penrose generalized inverse (DMPGI) is proved. 
Similar methods are then used to provide the Cholesky decomposition of dual real symmetric positive definite matrices. 
Both of our decompositions are driven by applications in numerical linear algebra.
\end{abstract}

\bigskip


\newpage

\section{Introduction}\label{sec: introduction} 



Dual numbers have been extensively utilized in a variety of fields, including robotics, 3D rigid body motion modeling, automatic differentiation, and neuroscience \cite{baydin2018automatic,kavan2006dual,perez2004dual,qi2023eigenvalues,wang2023sar,xu2024cur,xu2024randomized}. 
Many studies have utilized the orthogonal factorization of dual matrices to develop methods for identifying traveling wave patterns in the human brain \cite{xu2024qr,xu2024utv}.
Dual matrices also provide a theoretical foundation for pose transformation in robotics applications \cite{GU1987dual}.
The theoretical framework of numerical linear algebra involving dual numbers has been a subject of interest due to its implications across these diverse disciplines \cite{gutin2022generalizations,liu2024newsvendor,luo2022evaluating,pennestri2009linear,qi2022dual,wang2023crucial,zhu2022quartic}.


Complex symmetric matrices are mathematical tools frequently encountered in physics, for example, Hankel matrices in signal processing are complex symmetric matrices \cite{wang2017complex}. 
The Autoone-Takagi decomposition is an effective decomposition method for complex symmetric matrices and has been used in extensive numerical computation research in the past \cite{che2018adaptive}.
The LU decomposition is a significant numerical algorithm for solving linear systems of equations \cite{toledo1997locality}. 
It decomposes a matrix into a product of a lower triangular matrix and an upper triangular matrix and has been extended to the randomized LU decomposition for handling large-scale matrix problems \cite{shabat2018randomized}. 

The establishment of the Autoone-Takagi decomposition and LU decomposition of dual matrices is crucial for future applications.
This paper uses matrix transformation to obtain the Autoone-Takagi decomposition of dual complex symmetric matrices. 
By extending the properties of complex symmetry to the $\eta$-Hermitian properties of quaternions, as proposed by Horn and Zhang \cite{horn2012generalization}, we derive the generalized Autoone-Takagi decomposition of dual quaternion $\eta$-Hermitian matrices.
For the LU decomposition of dual matrices (DLU), the explicit expression of DLU decomposition can be obtained by solving the general solution of the Sylvester equation. 
Moreover, it can be proven that the existence of DLU decomposition is equivalent to dual rank-$k$ decomposition and dual Moore-Penrose generalized inverse (DMPGI).
With similar ideas, we can also obtain the Cholesky decomposition of dual real symmetric positive definite matrices.


The remaining content of this paper is organized as follows, Section \ref{sec: Preliminaries} introduces the basic knowledge of dual numbers and quaternions.
Section \ref{sec: Autonne-Takagi factorization of Dual Complex Matrices} presents the Autoone-Takagi decomposition of dual complex symmetric matrices and extends it to the generalized Autoone-Takagi decomposition of dual quaternion $\eta$-Hermitian matrices.
Autoone-Takagi type dual SVD (ATDSVD) is given for dual real symmetric matrices. 
In Section \ref{sec: LU Decomposition of Dual Matrix}, the LU decomposition of dual matrices is obtained by solving the Sylvester equation. 
The conditions for the existence of dual LU decomposition and the equivalence of dual rank-$k$ decomposition and dual Moore-Penrose generalized inverse are then proved.
Using the same approach, the Cholesky decomposition of dual real symmetric positive definite matrices (DCholesky) is presented.
Section \ref{sec: experiment} constructs a real symmetric positive definite dual matrix and computes the results of ATDSVD and DCholesky algorithms. 
It also provides an accurate error estimation for a noisy Hankel matrix as dual matrix.

\section{Preliminaries}\label{sec: Preliminaries}
This section introduces the basic concepts of dual numbers and quaternions used in the text, with notation consistent with MATLAB.
\subsection{Dual Numbers}\label{sec: Dual Numbers}
Consider a polynomial ring $\mathbb{F}$ over the real number, such as $\mathbb{R}, \mathbb{C}, \mathbb{H}$.
A dual number $d$ defined over $\mathbb{F}$ is denoted as $d =d_{s} + \epsilon d_{i}$, where $\epsilon$ represents the infinitesimal unit, defined by {\color{red}$\epsilon^{2} = 0$,} and $d_{s},d_{i}\in \mathbb{F}$ are the standard and infinitesimal parts of $d$, respectively. 
{\color{red}
Let $\mathbb{DF}$ be the set of dual numbers over $\mathbb{F}$.}
The operations of addition and multiplication for dual numbers $a$ and $b$, represented as $a=a_{s}+\epsilon a_{i}$ and $b=b_{s}+\epsilon b_{i}$, respectively, can be defined as follows,
\begin{equation}
    \begin{cases}
        a+b & = (a_{s}+b_{s})+\epsilon (a_{i}+b_{i}), \\
        ab & = (a_{s}b_{s}) + \epsilon (a_{i}b_{s}+a_{s}b_{i}).
    \end{cases}
\end{equation}
The dual number $a +b \epsilon$ can be represented by the square matrix $\begin{bmatrix}
    a&b\\
    0&a
\end{bmatrix}$. 
In this representation the matrix $\begin{bmatrix}
    0&1\\
    0&0
\end{bmatrix}$ squares to the zero matrix, corresponding to the dual number. 
{\color{red}
Let $\mathbb{DF}^{m \times n}$ be the set of $m \times n$ dual matrices over $\mathbb{F}$.
Dual matrices $A\in \mathbb{DF}^{m \times n}$ can be defined as $A_{s}+\epsilon A_{i}$, where $A_{s},A_{i}\in \mathbb{F}^{m \times n}$, and $A^{\top} = A_{s}^{\top}+\epsilon A_{i}^{\top}$. 
}
If a dual square matrix $A = A_{s}+\epsilon A_{i}\in \mathbb{DF}^{m \times n}$ is unitary, then it satisfies $A_{s}A_{s}^{\top} = I_{m}$ and $A_{s}A_{i}^{\top}+A_{i}A_{s}^{\top}=O_{m \times m}$.
\begin{definition}\label{def: SPD}
    A dual matrix $A=A_{s}+\epsilon A_{i}$ is symmetric positive definite iff $A_{s}$  is symmetric positive definite and $A_{i}$ is symmetric.
\end{definition}
\begin{definition}\label{def: DMPGI}
    \cite{wang2021characterizations} For a given dual matrix $A \in \mathbb{DR}^{m \times n}$ if there exists a dual matrix $X \in \mathbb{DR}^{n \times m}$ satisfying 
    \begin{equation}
        AXA=A,\quad XAX=X, \quad AX=(AX)^{\top}, \quad XA=(XA)^{\top}
    \end{equation}
    we call $X$ the dual Moore-Penrose generalized inverse(for short DMPGI) of $A$, and denote it as $A^{\dagger}$.
\end{definition}

\subsection{Quaternions}\label{sec: quaternion}

A quaternion $q \in \mathbb{H}$ can be expressed as $q = q_{0} + q_{1}\textbf{i} + q_{2}\textbf{j} + q_{3}\textbf{k}$, where $q_{0}, q_{1}, q_{2}, q_{3} \in \mathbb{R}$ and $\textbf{i,j,k}$ denote the three imaginary units of quaternions, meeting the following conditions,
\begin{equation}
    \textbf{i}^{2} =\textbf{j}^{2} =\textbf{k}^{2} =\textbf{ijk}=-1 .
\end{equation}
The quaternion $q$'s conjugate, denoted by $\bar{q}$, is defined as $\bar{q} = q_{0} - q_{1}\textbf{i} - q_{2}\textbf{j} - q_{3}\textbf{k}$.
Given a quaternion matrix $A \in \mathbb{H}^{m \times n}$, the conjugate operator $\overline{A}$, the transpose operator $A^{\top}$ and the conjugate transpose operator $A^{*}$ are given by $\overline{A} = (\overline{a}_{i,j})$, $A^{\top} = (a_{j,i})$, $A^{*} = \overline{A}^{\top} = \overline{A^{\top}} = (\overline{a}_{j,i})$.

\subsection{Autonne-Takagi Decomposition}\label{sec: Autonne-Takagi decomposition}
{\color{red}
The Autonne-Takagi decomposition \cite{autonne1915matrices,takagi1924algebraic} is a singular value decomposition for complex symmetric matrices \cite{horn2012matrix}. 
}
Specifically, any complex symmetric matrix $A\in \mathbb{C}^{n \times n}$ can be expressed as
\begin{equation}\label{equ: A = VsigmaV takagi}
    A = V\Sigma V^{\top},
\end{equation}
where $V\in \mathbb{C}^{n \times n}$ is a unitary matrix, and $\Sigma\in\mathbb{R}^{n \times n}$ is a nonnegative diagonal matrix.
This decomposition highlights the singular value properties of complex symmetric matrices, analogous to the orthogonal diagonalization of symmetric matrices in the real domain.
The Autonne-Takagi decomposition has wide applications in fields such as quantum information and signal processing \cite{teretenkov2022singular}.

\section{Autonne-Takagi Decomposition of Dual Complex Symmetric Matrices (DTakagi)}\label{sec: Autonne-Takagi factorization of Dual Complex Matrices}

In this section, we present the Autonne-Takagi decomposition of dual complex symmetric matrices (DTakagi).

\begin{theorem}\label{the: dualcomplexsymmetric} 
(DTakagi) For a dual complex symmetric matrix $A = A_{s} + \epsilon A_{i} \in \mathbb{DC}^{n\times n}$, we can perform the decomposition as follows,
\begin{equation}\label{equ: A = VsigmaV}
    A = V\Sigma V^{\top},
\end{equation}
where $\Sigma = \Sigma_{s}+\epsilon \Sigma_{i}\in  \mathbb{DR}^{n\times n}$ is a dual diagonal matrix, and $V = V_{s}+\epsilon V_{i}\in \mathbb{DC}^{n\times n} $ is unitary.
\end{theorem}
\begin{proof}
    Suppose $A= A_{s}+\epsilon A_{i}$ is a dual complex symmetric matrix, where $A_{s},A_{i} \in \mathbb{C}^{n \times n}$ are complex symmetric matrices. 
    According to the Autonne-Takagi factorization \cite{autonne1915matrices,takagi1924algebraic}, we can find a complex unitary matrix $W\in \mathbb{C}^{n \times n}$ such that $W^{\top}A_{s}W=D =  \diag(\lambda_{1}I_{n_{1}},\lambda_{2}I_{n_{2}},\dots,\lambda_{r}I_{n_{r}})\in \mathbb{R}^{n \times n}$ is diagonal where $\lambda_{1},\lambda_{2},\ldots,\lambda_{r}$ are all distinct and $n_{1}+n_{2}\cdots+n_{r}=n$. 
    So we can rewrite $W^{\top}AW = W^{\top}A_{s}W+\epsilon W^{\top}A_{i}W = D+\epsilon B\in \mathbb{DC}^{n\times n}$ as a block matrix
    \begin{equation}
        \begin{bmatrix}
            \lambda_{1}I_{n_{1}}&O &\cdots & O \\
            O & \lambda_{2}I_{n_{2}} &\ddots &\vdots\\
            \vdots & \ddots &\ddots & O\\
            O &\cdots & O & \lambda_{r}I_{n_{r}}
        \end{bmatrix}+\epsilon
        \begin{bmatrix}
            B_{11}& B_{12} &\cdots &  B_{1r} \\
             B_{12}^{\top} & B_{22} &\ddots &\vdots\\
            \vdots & \ddots &\ddots & B_{r-1,r}\\
           B_{1r}^{\top} &\cdots &  B_{r-1,r}^{\top} &  B_{rr}
        \end{bmatrix},
    \end{equation}
    where $B=W^{\top}A_{i}W$, each $B_{ii}(i =1,2,\ldots,r)$ is complex symmetric. 
    Next, we construct unitary matrices to iteratively diagonalize the infinitesimal part.
    We choose the unitary matrix
    \begin{equation}
        I+\epsilon Q =\begin{bmatrix}
            I_{n_{1}}&O &\cdots & O \\
            O & I_{n_{2}} &\ddots &\vdots\\
            \vdots & \ddots &\ddots & O\\
            O &\cdots & O & I_{n_{r}}
        \end{bmatrix}+\epsilon \begin{bmatrix}
            O&\frac{B_{12}}{\lambda_{1}-\lambda_{2}} &\cdots & \frac{ B_{1r}}{\lambda_{1}-\lambda_{r}} \\
            -\frac{ B_{12}^{\top}}{\lambda_{1}-\lambda_{2}} & O &\ddots &\vdots\\
            \vdots & \ddots &\ddots & \frac{ B_{r-1,r}}{\lambda_{r-1}-\lambda_{r}}\\
            -\frac{B_{1r}^{\top}}{\lambda_{1}-\lambda_{r}} &\cdots & -\frac{ B_{r-1,r}^{\top}}{\lambda_{r-1}-\lambda_{r}} & O
        \end{bmatrix}\in\mathbb{DC}^{n\times n}
    \end{equation}
    to block diagonalize the infinitesimal part of matrix $B$, while ensuring that the standard part $D$ remains unchanged.
    {\color{red}
    Through
    \begin{equation}\label{equ: (I+Q)(D+B)(I+QT)}
        \begin{aligned}
            &(I+\epsilon Q)W^{\top}AW(I+\epsilon Q)^{\top} = (I+\epsilon Q)(W^{\top}A_{s}W+\epsilon W^{\top}A_{i}W)(I+\epsilon Q)^{\top} \\
            &= (I+\epsilon Q)(D+\epsilon B)(I+\epsilon Q)^{\top} = D+\epsilon(QD+B+DQ^{\top})\\
             &=\begin{bmatrix}
            \lambda_{1}I_{n_{1}}&O &\cdots & O \\
            O & \lambda_{2}I_{n_{2}} &\ddots &\vdots\\
            \vdots & \ddots &\ddots & O\\
            O &\cdots & O & \lambda_{r}I_{n_{r}}
        \end{bmatrix}+ \epsilon \left(\begin{bmatrix}
            O&\frac{\lambda_{2}B_{12}}{\lambda_{1}-\lambda_{2}} &\cdots & \frac{ \lambda_{r}B_{1r}}{\lambda_{1}-\lambda_{r}} \\
            -\frac{ \lambda_{1}B_{12}^{\top}}{\lambda_{1}-\lambda_{2}} & O &\ddots &\vdots\\
            \vdots & \ddots &\ddots & \frac{ \lambda_{r}B_{r-1,r}}{\lambda_{r-1}-\lambda_{r}}\\
            -\frac{\lambda_{1}B_{1r}^{\top}}{\lambda_{1}-\lambda_{r}} &\cdots & -\frac{ \lambda_{r-1}B_{r-1,r}^{\top}}{\lambda_{r-1}-\lambda_{r}} & O
            \end{bmatrix}\right.\\
           &\left.+\begin{bmatrix}
            B_{11}& B_{12} &\cdots &  B_{1r} \\
             B_{12}^{\top} & B_{22} &\ddots &\vdots\\
            \vdots & \ddots &\ddots & B_{r-1,r}\\
           B_{1r}^{\top} &\cdots &  B_{r-1,r}^{\top} &  B_{rr}
        \end{bmatrix}+\begin{bmatrix}
            O&-\frac{\lambda_{1}B_{12}}{\lambda_{1}-\lambda_{2}} &\cdots & -\frac{\lambda_{1}B_{1r}}{\lambda_{1}-\lambda_{r}} \\
            \frac{\lambda_{2}B_{12}^{\top}}{\lambda_{1}-\lambda_{2}} & O &\ddots &\vdots\\
            \vdots & \ddots &\ddots & -\frac{\lambda_{r-1}B_{r-1,r}}{\lambda_{r-1}-\lambda_{r}}\\
            \frac{\lambda_{r}B_{1r}^{\top}}{\lambda_{1}-\lambda_{r}} &\cdots & \frac{\lambda_{r}B_{r-1,r}^{\top}}{\lambda_{r-1}-\lambda_{r}} & O
        \end{bmatrix}\right)\\
        &=\begin{bmatrix}
            \lambda_{1}I_{n_{1}}&O &\cdots & O \\
            O & \lambda_{2}I_{n_{2}} &\ddots &\vdots\\
            \vdots & \ddots &\ddots & O\\
            O &\cdots & O & \lambda_{r}I_{n_{r}}
        \end{bmatrix}+ \epsilon \begin{bmatrix}
            B_{11}& O &\cdots &  O \\
             O & B_{22} &\ddots &\vdots\\
            \vdots & \ddots &\ddots & O\\
           O &\cdots &  O &  B_{rr}
        \end{bmatrix}
        \end{aligned},
    \end{equation}
   we can understand the entire diagonalization process.}
    Considering that the diagonal blocks $B_{ii}$ are all complex symmetric matrices, there exists the Autonne-Takagi decomposition $B_{ii} = U_{i}\Theta_{i}U_{i}^{\top}$, where $U_{i}$ is a unitary matrix, and $\Theta_{i}$ is a real nonnegative diagonal matrix. 
    Thus, $U= \diag(U_{1},U_{2},\ldots,U_{r})\in\mathbb{C}^{n \times n}$ can diagonalize $B_{ii}$.
    By organizing the above process, we obtain the following $\Sigma = \Sigma_{s}+\epsilon \Sigma_{i}\in \mathbb{DR}^{n\times n}$, where $\Sigma_{s} = \diag(\lambda_{1}I_{n_{1}},\lambda_{2}I_{n_{2}},\dots,\lambda_{r}I_{n_{r}})\in\mathbb{R}^{n\times n}$ and $\Sigma_{i} = \diag(\Theta_{1},\Theta_{2},\ldots,\Theta_{r})\in\mathbb{R}^{n\times n}$.
    {\color{red}
    By combining \eqref{equ: (I+Q)(D+B)(I+QT)} with the above process, we can obtain
    \begin{equation}\label{equ: UsigmaU }
        \begin{aligned}
        (I+\epsilon Q)W^{\top}AW(I+\epsilon Q)^{\top} &= U\Sigma U^{\top}, \\
        U^{\top}(I+\epsilon Q)W^{\top}AW(I+\epsilon Q)^{\top}U & = \Sigma ,\\
        (U^{\top}W^{\top}+\epsilon U^{\top}QW^{\top})A(WU+\epsilon WQ^{\top}U)& =\Sigma ,\\
        (U^{\top}W^{\top}-\epsilon U^{\top}Q^{\top}W^{\top})A(WU-\epsilon WQU)& =\Sigma ,
        \end{aligned}
    \end{equation}
    where $Q$ is skew-symmetric which means $Q+Q^{\top} = O$.}
    Let $V = V_{s}+\epsilon V_{i} \in \mathbb{DC}^{n\times n}$, where $V_{s}=WU\in \mathbb{C}^{n\times n}$ and $V_{i}=-WQU\in \mathbb{C}^{n\times n}$.
    Since that
    \begin{equation}\label{equ: VVtop =I}
        \begin{aligned}
        VV^{\top} & = (V_{s}+\epsilon V_{i}) (V_{s}^{\top}+\epsilon V_{i}^{\top}) \\
        &= (WU-\epsilon WQU)(U^{\top}W^{\top}-\epsilon U^{\top}Q^{\top}W^{\top}) \\
        &= WW^{\top}-\epsilon (WQW^{\top}+WQ^{\top}W^{\top})\\
        &=I_{n},
        \end{aligned}
    \end{equation}
    where $Q$ is skew-symmetric which means $Q+Q^{\top} = O$,  we achieve the desired form.
\end{proof}

We provide pseudocode for DTakagi based on the proof from Theorem \ref{the: dualcomplexsymmetric}, and present it in Algorithm \ref{alg: autonne-takagi}.

\begin{algorithm}[htb]
\DontPrintSemicolon
    \KwInput{Dual complex symmetric matrix $A = A_{s}+A_{i}\epsilon \in \mathbb{DC}^{n \times n}$.}
    
    Compute the Autonne-Takagi factorization of $A_{s} = WDW^{\top}$ .

    Compute $B= W^{\top}A_{i}W =\begin{bmatrix}
            B_{11}& B_{12} &\cdots &  B_{1r} \\
             B_{12}^{\top} & B_{22} &\ddots &\vdots\\
            \vdots & \ddots &\ddots & B_{r-1,r}\\
           B_{1r}^{\top} &\cdots &  B_{r-1,r}^{\top} &  B_{rr}
        \end{bmatrix}$.

    Compute $ Q = \begin{bmatrix}
            O&\frac{B_{12}}{\lambda_{1}-\lambda_{2}} &\cdots & \frac{ B_{1r}}{\lambda_{1}-\lambda_{r}} \\
            -\frac{ B_{12}^{\top}}{\lambda_{1}-\lambda_{2}} & O &\ddots &\vdots\\
            \vdots & \ddots &\ddots & \frac{ B_{r-1,r}}{\lambda_{r-1}-\lambda_{r}}\\
            -\frac{B_{1r}^{\top}}{\lambda_{1}-\lambda_{r}} &\cdots & -\frac{ B_{r-1,r}^{\top}}{\lambda_{r-1}-\lambda_{r}} & O
        \end{bmatrix}$.

    Compute the Autonne-Takagi factorization of $B_{ii} = U_{i}\Theta_{i}U_{i}^{\top}, i =1,2,\ldots,r$. 
    
    Set $\Sigma_{s} = D$ and $\Sigma_{i} = \diag(\Theta_{1},\Theta_{2},\ldots,\Theta_{r})$.

    Set $V_{s}=WU$ and $V_{i} = -WQU$, where $U= \diag(U_{1},U_{2},\ldots,U_{r})$.

    \KwOutput{$V = V_{s}+V_{i}\epsilon \in \mathbb{DC}^{n \times n}$ is unitary, and $\Sigma = \Sigma_{s}+\Sigma_{i}\epsilon \in \mathbb{DR}^{n \times n}$ is a dual diagonal matrix.}
    \caption{Autonne-Takagi Decomposition of the Dual Complex Symmetric Matrix (DTakagi)}
    \label{alg: autonne-takagi}
\end{algorithm}

\subsection{Dual Quaternion \texorpdfstring{$\eta$}.-Hermitian Matrix}\label{sec: Dual Quaternion eta-Hermitian Matrix}
This section extends the concept of complex symmetry to $\eta$-Hermitian in quaternions and generalizes the Autonne-Takagi decomposition of dual complex symmetric matrices to the decomposition of dual quaternion $\eta$-Hermitian matrix.
Consider the quaternion field $\mathbb{H}$, it has three imaginary parts {\color{red}$\textbf{i},\textbf{j}$, and $\textbf{k}$. }
There is an important identity that is 
\begin{equation}\label{equ: -jzj=-kzk=bar(z)}
    -\textbf{j}z\textbf{j} = -\textbf{k}z\textbf{k} = \bar{z},\text{ for each complex number } z=a+b\textbf{i},
\end{equation}
where $a,b \in \mathbb{R}$.
The identity demonstrates that by using the unit imaginary quaternions $\textbf{j}$ and $\textbf{k}$, each complex number can be made similar to its conjugate over the quaternions.
The literature \cite{horn2012generalization} used this characteristic to define the $\eta$-Hermitian property of the quaternion matrix.
Let $\eta$ be a given unit pure imaginary quaternion, then $\eta^{H} = \bar{\eta}  = -\eta, \eta^{2} = -1$. 
\begin{definition}\label{def: eta-Hermitian}
    ($\eta$-Hermitian) For a quaternion matrix $A\in\mathbb{H}^{n\times n}$, we define
\begin{equation}\label{equ: Aeta=etaAeta}
    A^{\eta} = \eta^{H}A\eta, A^{H} = \bar{A}^{\top}\textit{(conjugate transpose), and } A^{\eta H} = \eta^{H} A^{H} \eta. 
\end{equation}
So a quaternion $n\times n$ matrix $A$ is said to be $\eta$-Hermitian if $A = A^{\eta H}$.
\end{definition}
If $A\in \mathbb{H}^{n\times n}$ has all complex entries, simple calculations can obtain that
{\color{red}
\begin{equation}\label{equ: AiH=AH}
    A^{\textbf{i}H} = A^{H}, \textit{ and } A^{\textbf{j}H} = A^{\textbf{k}H} = (\bar{A})^{H} = A^{\top}.
\end{equation}
}
Let $B\in \mathbb{R}^{n \times n}$. 
Then $B$ is a symmetric matrix if and only if it is $\eta$-Hermitian for any unit pure imaginary $\eta$.
Let $C\in \mathbb{C}^{n \times n}$.
{\color{red}
Then $C$ is a Hermitian matrix if and only if it is $\textbf{i}$-Hermitian, while $C$ is a symmetric matrix if and only if it is both $\textbf{j}$-Hermitian and $\textbf{k}$-Hermitian \cite{horn2012generalization}. 
}
Now we introduce the $\eta$-Hermitian property of quaternions, which extends the complex symmetric property.
 The following theorem can achieve the generalization of the complex Autonne–Takagi factorization to quaternion matrices.
\begin{theorem}[\cite{horn2012generalization}, Theorem 3]\label{the: eta_Hermitian}
Let $\eta$ be a unit pure imaginary quaternion and suppose that $A\in \mathbb{H}^{n\times n}$ is $\eta$-Hermitian. 
Then there is a unitary quaternion matrix $U$ and a real nonnegative diagonal matrix $\Sigma$ such that
\begin{equation}\label{equ:A=UsigmaUh}
    A = U\Sigma U^{\eta H}
\end{equation}
The diagonal entries of $\Sigma$ are the singular values of $A$.
\end{theorem}
Now consider the case of dual quaternion matrices. 
We generalize the complex Autonne-Takagi decomposition to dual quaternion $\eta$-Hermitian matrices. 
\begin{theorem}\label{the: dualquaternionsymmetric}
For a dual quaternion matrix $A=A_{s}+\epsilon A_{i}\in \mathbb{DH}^{n \times n}$ which is $\eta$-Hermitian, we can perform the decomposition as follows,
\begin{equation}\label{equ: A=VsigmaVh}
    A = V\Sigma V^{\eta H}
\end{equation}
where $\Sigma = \Sigma_{s}+\epsilon \Sigma_{i}\in  \mathbb{DR}^{n\times n}$ is a dual diagonal matrix, and $V=V_{s}+\epsilon V_{i} \in\mathbb{DH}^{n\times n}$ is unitary.
\end{theorem}
\begin{proof}  
    Suppose $A=A_{s}+\epsilon A_{i}$ is a dual quaternion $\eta$-Hermitian matrix, where $A_{s},A_{i}\in \mathbb{H}^{n \times n}$. 
    According to the Theorem \ref{the: eta_Hermitian}, we can find a quaternion unitary matrix $W \in \mathbb{H}^{n \times n}$ such that $W^{\eta H}A_{s}W=D=  \diag(\lambda_{1}I_{n_{1}},\lambda_{2}I_{n_{2}},\dots,\lambda_{r}I_{n_{r}})\in \mathbb{R}^{n \times n}$ is diagonal where $\lambda_{1},\lambda_{2},\ldots,\lambda_{r}$ are all distinct and $n_{1}+n_{2}\cdots+n_{r}=n$. 
    So we can rewrite $W^{\eta H}AW = W^{\eta H}A_{s}W+\epsilon W^{\eta H}A_{i}W = D+\epsilon B\in \mathbb{DH}^{n\times n}$ as a block matrix
    \begin{equation}
        \begin{bmatrix}
            \lambda_{1}I_{n_{1}}&O &\cdots & O \\
            O & \lambda_{2}I_{n_{2}} &\ddots &\vdots\\
            \vdots & \ddots &\ddots & O\\
            O &\cdots & O & \lambda_{r}I_{n_{r}}
        \end{bmatrix}+\epsilon
        \begin{bmatrix}
            B_{11}& B_{12} &\cdots &  B_{1r} \\
             B_{12}^{\eta H} & B_{22} &\ddots &\vdots\\
            \vdots & \ddots &\ddots & B_{r-1,r}\\
           B_{1r}^{\eta H} &\cdots &  B_{r-1,r}^{\eta H} &  B_{rr}
        \end{bmatrix}
    \end{equation}
    where $B=W^{\eta H}A_{i}W$, each $B_{ii}(i =1,2,\ldots,r)$ is $\eta$-Hermitian. 
    Next, we construct unitary matrices to diagonalize the infinitesimal part iteratively.
    We choose the unitary matrix
    \begin{equation}
        I+\epsilon Q =\begin{bmatrix}
            I_{n_{1}}&O &\cdots & O \\
            O & I_{n_{2}} &\ddots &\vdots\\
            \vdots & \ddots &\ddots & O\\
            O &\cdots & O & I_{n_{r}}
        \end{bmatrix}+\epsilon \begin{bmatrix}
            O&\frac{B_{12}}{\lambda_{1}-\lambda_{2}} &\cdots & \frac{B_{1r}}{\lambda_{1}-\lambda_{r}} \\
            -\frac{ B_{12}^{\eta H}}{\lambda_{1}-\lambda_{2}} & O &\ddots &\vdots\\
            \vdots & \ddots &\ddots & \frac{ B_{r-1,r}}{\lambda_{r-1}-\lambda_{r}}\\
            -\frac{B_{1r}^{\eta H}}{\lambda_{1}-\lambda_{r}} &\cdots & -\frac{ B_{r-1,r}^{\eta H}}{\lambda_{r-1}-\lambda_{r}} & O
        \end{bmatrix}\in\mathbb{DH}^{n\times n}
    \end{equation}
    to block diagonalize the infinitesimal part of matrix $B$, while ensuring that the standard part $D$ remains unchanged.
    {\color{red}
    Through
    \begin{equation}\label{equ: (I+Q)(D+B)(I+Qeta H)}
        \begin{aligned}
            &(I+\epsilon Q)W^{\eta H}AW(I+\epsilon Q)^{\eta H} = (I+\epsilon Q)(W^{\eta H}A_{s}W+\epsilon W^{\eta H}A_{i}W)(I+\epsilon Q)^{\eta H} \\
            &= (I+\epsilon Q)(D+\epsilon B)(I+\epsilon Q)^{\eta H} = D+\epsilon(QD+B+DQ^{\eta H})\\
             &=\begin{bmatrix}
            \lambda_{1}I_{n_{1}}&O &\cdots & O \\
            O & \lambda_{2}I_{n_{2}} &\ddots &\vdots\\
            \vdots & \ddots &\ddots & O\\
            O &\cdots & O & \lambda_{r}I_{n_{r}}
        \end{bmatrix}+ \epsilon \left(\begin{bmatrix}
            O&\frac{\lambda_{2}B_{12}}{\lambda_{1}-\lambda_{2}} &\cdots & \frac{ \lambda_{r}B_{1r}}{\lambda_{1}-\lambda_{r}} \\
            -\frac{ \lambda_{1}B_{12}^{\eta H}}{\lambda_{1}-\lambda_{2}} & O &\ddots &\vdots\\
            \vdots & \ddots &\ddots & \frac{ \lambda_{r}B_{r-1,r}}{\lambda_{r-1}-\lambda_{r}}\\
            -\frac{\lambda_{1}B_{1r}^{\eta H}}{\lambda_{1}-\lambda_{r}} &\cdots & -\frac{ \lambda_{r-1}B_{r-1,r}^{\eta H}}{\lambda_{r-1}-\lambda_{r}} & O
            \end{bmatrix}\right.\\
           &\left.+\begin{bmatrix}
            B_{11}& B_{12} &\cdots &  B_{1r} \\
             B_{12}^{\eta H} & B_{22} &\ddots &\vdots\\
            \vdots & \ddots &\ddots & B_{r-1,r}\\
           B_{1r}^{\eta H} &\cdots &  B_{r-1,r}^{\eta H} &  B_{rr}
        \end{bmatrix}+\begin{bmatrix}
            O&-\frac{\lambda_{1}B_{12}}{\lambda_{1}-\lambda_{2}} &\cdots & -\frac{\lambda_{1}B_{1r}}{\lambda_{1}-\lambda_{r}} \\
            \frac{\lambda_{2}B_{12}^{\eta H}}{\lambda_{1}-\lambda_{2}} & O &\ddots &\vdots\\
            \vdots & \ddots &\ddots & -\frac{\lambda_{r-1}B_{r-1,r}}{\lambda_{r-1}-\lambda_{r}}\\
            \frac{\lambda_{r}B_{1r}^{\eta H}}{\lambda_{1}-\lambda_{r}} &\cdots & \frac{\lambda_{r}B_{r-1,r}^{\eta H}}{\lambda_{r-1}-\lambda_{r}} & O
        \end{bmatrix}\right)\\
        &=\begin{bmatrix}
            \lambda_{1}I_{n_{1}}&O &\cdots & O \\
            O & \lambda_{2}I_{n_{2}} &\ddots &\vdots\\
            \vdots & \ddots &\ddots & O\\
            O &\cdots & O & \lambda_{r}I_{n_{r}}
        \end{bmatrix}+ \epsilon \begin{bmatrix}
            B_{11}& O &\cdots &  O \\
             O & B_{22} &\ddots &\vdots\\
            \vdots & \ddots &\ddots & O\\
           O &\cdots &  O &  B_{rr}
        \end{bmatrix}
        \end{aligned},
    \end{equation}
   we can understand the entire diagonalization process.}
    Considering that the diagonal blocks $B_{ii}$ are all $\eta$-Hermitian matrices, according to Theorem \ref{the: eta_Hermitian} there exists decomposition $B_{ii} = U_{i}\Theta_{i}U_{i}^{\eta H}$, where $U_{i}$ is a quaternion unitary matrix, and $\Theta_{i}$ is a real nonnegative diagonal matrix. 
    Thus, $U= \diag(U_{1},U_{2},\ldots,U_{r})\in\mathbb{H}^{n \times n}$ can diagonalize $B_{ii}$.
    By organizing the above process, we obtain the following $\Sigma = \Sigma_{s}+\epsilon \Sigma_{i}\in \mathbb{DR}^{n\times n}$, where $\Sigma_{s} = \diag(\lambda_{1}I_{n_{1}},\lambda_{2}I_{n_{2}},\dots,\lambda_{r}I_{n_{r}})\in\mathbb{R}^{n\times n}$ and $\Sigma_{i} = \diag(\Theta_{1},\Theta_{2},\ldots,\Theta_{r})\in\mathbb{R}^{n\times n}$.
    {\color{red}
    By combining \eqref{equ: (I+Q)(D+B)(I+Qeta H)} with the above process, we can obtain
    \begin{equation}\label{equ: UsigmaUeta }
        \begin{aligned}
        (I+\epsilon Q)W^{\eta H}AW(I+\epsilon Q)^{\eta H} &= U\Sigma U^{\eta H}, \\
        U^{\eta H}(I+\epsilon Q)W^{\eta H}AW(I+\epsilon Q)^{\eta H}U & = \Sigma ,\\
        (U^{\eta H}W^{\eta H}+\epsilon U^{\eta H}QW^{\eta H})A(WU+\epsilon WQ^{\eta H}U)& =\Sigma ,\\
        (U^{\eta H}W^{\eta H}-\epsilon U^{\eta H}Q^{\eta H}W^{\eta H})A(WU-\epsilon WQU)& =\Sigma ,
        \end{aligned}
    \end{equation}
    where $Q$ satifies $Q+Q^{\eta H} = O$.}
    Let $V = V_{s}+\epsilon V_{i} \in \mathbb{DH}^{n\times n}$, where $V_{s}=WU\in \mathbb{H}^{n\times n}$ and $V_{i}=-WQU\in \mathbb{H}^{n\times n}$.
    Since that
    \begin{equation}\label{equ: VVetaH =I}
        \begin{aligned}
        VV^{\eta H} & = (V_{s}+\epsilon V_{i}) (V_{s}^{\eta H}+\epsilon V_{i}^{\eta H}) \\
        &= (WU-\epsilon WQU)(U^{\eta H}W^{\eta H}-\epsilon U^{\eta H}Q^{\eta H}W^{\eta H}) \\
        &= WW^{\eta H}-\epsilon (WQW^{\eta H}+WQ^{\eta H}W^{\eta H})\\
        &=I_{n},
        \end{aligned}
    \end{equation}
    where $Q$ satifies $Q+Q^{\eta H} = O$,  we achieve the desired form.
\end{proof}

The factorization of the $\eta$-Hermitian matrix is the generalization of the Autonne–Takagi factorization of the dual complex symmetric matrix.
Here we consider a $2\times 2$ dual real symmetric matrix $A$ as a quaternion matrix, where
\begin{equation}\label{equ:exampleA}
    A = \begin{bmatrix}
        0&1\\
        1&0
    \end{bmatrix}+\epsilon \begin{bmatrix}
        0&1\\
        1&0
    \end{bmatrix}.
\end{equation}
{\color{red}
Obviously, $A$ is both $\textbf{i}$-Hermitian and $\textbf{j}$-Hermitian. 
Utilizing the Theorem \ref{the: dualquaternionsymmetric}, we can achieve that $A = U\Sigma_{1}U^{\textbf{i}H}$, where $U$ is a quaternion unitary matrix
\begin{displaymath}
    U = \frac{1}{\sqrt{2}}\begin{bmatrix}
        1&\textbf{j}\\
        1&-\textbf{j}
    \end{bmatrix}, \Sigma_{1} = \begin{bmatrix}
        1+\epsilon & 0\\
        0& 1+\epsilon
    \end{bmatrix}.
\end{displaymath}
Also $A = V\Sigma_{2}V^{\textbf{j}H} = V\Sigma_{2}V^{\top}$, where $V$ is a complex unitary matrix
\begin{displaymath}
    V = \frac{1}{2}\begin{bmatrix}
        1+\textbf{i}&1-\textbf{i}\\
        1-\textbf{i}&1+\textbf{i}
    \end{bmatrix}, \Sigma_{2} = \begin{bmatrix}
        1+\epsilon & 0\\
        0& 1+\epsilon
    \end{bmatrix}.
\end{displaymath}
}
Using the conclusion from Theorem \ref{the: dualquaternionsymmetric}, we provide pseudocode for the Autonne-Takagi decomposition of dual quaternion $\eta$-Hermitian matrices in Algorithm \ref{alg: takagi quaternion}.

\begin{algorithm}[htb]
\DontPrintSemicolon
    \KwInput{Dual quaternion $\eta$-Hermitian matrix $A = A_{s}+A_{i}\epsilon \in \mathbb{DH}^{n \times n}$.}
    
    Compute the Autonne-Takagi factorization of $A_{s} = WDW^{\eta H}$ .

    Compute $B= W^{\eta H}A_{i}W =\begin{bmatrix}
            B_{11}& B_{12} &\cdots &  B_{1r} \\
             B_{12}^{\eta H} & B_{22} &\ddots &\vdots\\
            \vdots & \ddots &\ddots & B_{r-1,r}\\
           B_{1r}^{\eta H} &\cdots &  B_{r-1,r}^{\eta H} &  B_{rr}
        \end{bmatrix}$.

    Compute $ Q = \begin{bmatrix}
            O&\frac{B_{12}}{\lambda_{1}-\lambda_{2}} &\cdots & \frac{B_{1r}}{\lambda_{1}-\lambda_{r}} \\
            -\frac{ B_{12}^{\eta H}}{\lambda_{1}-\lambda_{2}} & O &\ddots &\vdots\\
            \vdots & \ddots &\ddots & \frac{ B_{r-1,r}}{\lambda_{r-1}-\lambda_{r}}\\
            -\frac{B_{1r}^{\eta H}}{\lambda_{1}-\lambda_{r}} &\cdots & -\frac{ B_{r-1,r}^{\eta H}}{\lambda_{r-1}-\lambda_{r}} & O
        \end{bmatrix}$.



    Compute the Autonne-Takagi factorization of $B_{ii} = U_{i}\Theta_{i}U_{i}^{\eta H}, i =1,2,\ldots,r$. 
    
    Set $\Sigma_{s} = D$ and $\Sigma_{i} = \diag(\Theta_{1},\Theta_{2},\ldots,\Theta_{r})$.

    Set $V_{s}=WU$ and $V_{i} = -WQU$, where $U= \diag(U_{1},U_{2},\ldots,U_{r})$.
    
    \KwOutput{$V = V_{s}+V_{i}\epsilon \in \mathbb{DH}^{n \times n}$ is unitary, and $\Sigma = \Sigma_{s}+\Sigma_{i}\epsilon \in \mathbb{DR}^{n \times n}$ is a dual diagonal matrix.}
    \caption{Autonne-Takagi Decomposition of the Dual Quaternion $\eta$-Hermitian Matrix}
    \label{alg: takagi quaternion}
\end{algorithm}

\begin{remark}\label{rem: Dcholesky}
    The Autonne-Takagi decomposition on real symmetric matrices is equivalent to SVD. 
    Therefore, using the content of this section, we can obtain the SVD of dual real symmetric matrices through the algorithm of Takagi decomposition.
    So we can get the Autoone-Takagi type dual SVD (ATDSVD), and pseudocode is given in Algorithm \ref{alg: ATDCD}.
\end{remark}

\begin{algorithm}[htb]
\DontPrintSemicolon
    \KwInput{Dual symmetric matrix $A = A_{s}+A_{i}\epsilon \in \mathbb{DR}^{n \times n}$.}
    
    Compute the SVD of $A_{s} = WDW^{\top}$ .

    Compute $B= W^{\top}A_{i}W =\begin{bmatrix}
            B_{11}& B_{12} &\cdots &  B_{1r} \\
             B_{12}^{\top} & B_{22} &\ddots &\vdots\\
            \vdots & \ddots &\ddots & B_{r-1,r}\\
           B_{1r}^{\top} &\cdots &  B_{r-1,r}^{\top} &  B_{rr}
        \end{bmatrix}$.

    Compute $ Q = \begin{bmatrix}
            O&\frac{B_{12}}{\lambda_{1}-\lambda_{2}} &\cdots & \frac{B_{1r}}{\lambda_{1}-\lambda_{r}} \\
            -\frac{ B_{12}^{\top}}{\lambda_{1}-\lambda_{2}} & O &\ddots &\vdots\\
            \vdots & \ddots &\ddots & \frac{ B_{r-1,r}}{\lambda_{r-1}-\lambda_{r}}\\
            -\frac{B_{1r}^{\top}}{\lambda_{1}-\lambda_{r}} &\cdots & -\frac{ B_{r-1,r}^{\top}}{\lambda_{r-1}-\lambda_{r}} & O
        \end{bmatrix}$.



    Compute the SVD of $B_{ii} = U_{i}\Theta_{i}U_{i}^{\top}, i =1,2,\ldots,r$. 
    
    Set $\Sigma_{s} = D$ and $\Sigma_{i} = \diag(\Theta_{1},\Theta_{2},\ldots,\Theta_{r})$.

    Set $V_{s}=WU$ and $V_{i} = -WQU$, where $U= \diag(U_{1},U_{2},\ldots,U_{r})$.

    \KwOutput{$\Sigma = \Sigma_{s}+\Sigma_{i}\epsilon \in \mathbb{DR}^{n \times n}$ is a dual real diagonal matrix, and $V = V_{s}+V_{i}\epsilon \in \mathbb{DR}^{n \times n}$ is a unitary dual matrix.}
    \caption{Autoone-Takagi type dual singular value decomposition (ATDSVD)}
    \label{alg: ATDCD}
\end{algorithm}

\section{LU Decomposition of Dual Matrices (DLU) }\label{sec: LU Decomposition of Dual Matrix}
In this section, we consider the LU decomposition of the dual matrix (DLU).
{\color{red}
Wang et al. \cite{wang2024algebraic} proposed a structure-preserving LU decomposition algorithm for dual quaternion matrices using an algebraic approach based on Gaussian transformations, which outlines the computational process for the LU decomposition of dual quaternion matrices. 
Our goal is to provide the expression and algorithm for the LU decomposition of dual matrices by leveraging the solution of the Sylvester equation.
}
Consider a square dual matrix $A\in A_{s}+A_{i}\epsilon \in \mathbb{DR}^{n \times n}$, assuming its full rank DLU decomposition $A = LU$, where $L = L_{s}+\epsilon L_{i}, U =U_{s}+\epsilon U_{i} \in \mathbb{DR}^{n \times n}$ exists. 
Then, using the matrix representation, we can obtain,
\begin{equation}\label{equ: LU matrix represention}
    \begin{bmatrix}
        A_{s} & A_{i} \\
        O & A_{s}
    \end{bmatrix} = \begin{bmatrix}
        L_{s} & L_{i} \\
        O & L_{s}
    \end{bmatrix}\begin{bmatrix}
        U_{s} & U_{i} \\
        O & U_{s}
    \end{bmatrix}
\end{equation}
This leads to two matrix equations,
\begin{equation}\label{equ: Ai = LsUi+LiUs}
    \begin{cases}
        A_{s} & = L_{s}U_{s}, \\
        A_{i} & = L_{s}U_{i} + L_{i}U_{s}.
    \end{cases}
\end{equation}
The first equation is the standard LU decomposition, and the second is a Sylvester equation. 
The idea of solving this system of matrix equations \eqref{equ: Ai = LsUi+LiUs} can be extended to the general rank-$k$ LU decomposition of dual matrices.
Here we first introduce a lemma for the general solution of the Sylvester equation.
\begin{lemma}\label{lem: solution of Sylvester equation}
    \cite{Baksalary1979TheME} Given $A\in \mathbb{R}^{m \times k}, B\in \mathbb{R}^{l \times n}$, and $C\in\mathbb{R}^{m \times n}$. The equation
    \begin{equation}\label{equ: Ax-YB =C}
        AX-YB = C,
    \end{equation}
    has a solution $X\in \mathbb{R}^{k \times n},Y \in \mathbb{R}^{m \times l}$ iff
    \begin{equation}
        (I_{m}-AA^{\dagger})C(I_{n}-B^{\dagger}B) = O_{m\times n}.
    \end{equation}
    If this is the case, then the general solution of (\ref{equ: Ax-YB =C}) has the form
    \begin{equation}
        \begin{cases}
        X = A^{\dagger}C+A^{\dagger}ZB+(I_{k}-A^{\dagger}A)W,\\
        Y = -(I_{m}-AA^{\dagger})CB^{\dagger}+Z-(I_{m}-AA^{\dagger})ZBB^{\dagger},
        \end{cases}
    \end{equation}
    with $W\in \mathbb{R}^{k\times n}$, $Z\in\mathbb{R}^{m\times l}$  being arbitrary and $A^\dagger$ representing the Moore-Penrose pseudoinverse of $A$.   
\end{lemma}
With the help of Lemma \ref{lem: solution of Sylvester equation}, we can solve the dual LU decomposition of dual matrices.
\begin{theorem}\label{the: DLU decomposition}
    (DLU Decomposition) Considering $A = A_{s} + A_{i}\epsilon \in \mathbb{DR}^{m \times n}$ with $m \geq n$. 
    Assume that $A_{s} = L_{s}U_{s}$ is the LU decomposition of $A_{s}$, where $k \leq \min\{m,n\}$, $L_{s} \in \mathbb{R}^{m \times k}$ is a lower trapezoidal matrix with rank $k$, and $U_{s} \in \mathbb{R}^{k \times n}$ is an upper trapezoidal matrix with rank $k$.
    Then the LU decomposition of the dual matrix $A =LU$, where $L = L_{s}+L_{i}\epsilon \in \mathbb{DR}^{m \times k} $ is a lower trapezoidal dual matrix, and $U = U_{s}+U_{i}\epsilon \in \mathbb{DR}^{k \times n}$ is an upper trapezoidal dual matrix exist iff
    \begin{equation}\label{equ: (I-LsLsT)Ai(I-UsUs)}
        (I_{m}-L_{s}L_{s}^{\dagger})A_{i}(I_{n}-U_{s}^{\dagger}U_{s}) = O_{m\times n}.
    \end{equation}
    Additionally, the expressions of $L_{i}$, $U_{i}$, the infinitesimal parts of $L,U$, are
    \begin{equation}\label{equ: Li Ui expression}
        \begin{cases}
            L_{i} & = (I_{m}-L_{s}L_{s}^{\dagger})A_{i}U_{s}^{\dagger} + L_{s}P, \\
            U_{i}& = L_{s}^{\dagger}A_{i}-PU_{s},
        \end{cases}
    \end{equation}
    where $P\in \mathbb{R}^{k \times k}$. 
    We can obtain the expression of the upper triangular part of $P$ as
    \begin{equation}
        \begin{cases}
            P(2:k,1) & =\frac{1}{U_{s}(1,1)}B(2:k,1),\\
            P(3:k,2) & = \frac{1}{U_{s}(2,2)}(B(3:k,2)-P(3:k,1)U_{s}(1,2)), \\
            &\ldots \\
            P(k,k-1) & = \frac{1}{U_{s}(k-1,k-1)}(B(k,k-1)-P(k,1:k-2)U_{s}(1:k-2,k-1)),
        \end{cases}
    \end{equation}
    and the expression of the lower triangular part of $P$ as
    \begin{equation}
        \begin{cases}
            P(1,2:k) & = \frac{1}{L_{s}(1,1)}(-C(1,2:k)),\\
            P(2,3:k) & =\frac{1}{L_{s}(2,2)}(-C(2,3:k)-L_{s}(2,1)P(1,3:k)),\\
            & \ldots \\
            P(k-1,k)&=\frac{1}{L_{s}(k-1,k-1)}(-C(k-1,k)-L_{s}(k-1,1:k-2)P(1:k-2,k)),
        \end{cases}
    \end{equation}
    where $B = L_{s}^{\dagger}A_{i} \in \mathbb{R}^{k \times n}$, and $C = (I_{m}-L_{s}L_{s}^{\dagger})A_{i}U_{s}^{\dagger} \in \mathbb{R}^{m\times k} $. 
\end{theorem}
\begin{proof}
    Assume that the LU decomposition of the dual matrix $A$ exists, the standard and infinitesimal parts of $A$ can be expressed as
    \begin{equation}\label{equ: Ai = LiUs+LsUi}
        \begin{cases}
            A_{s} & = L_{s}U_{s}, \\
            A_{i} & = L_{i}U_{s} + L_{s}U_{i}.
        \end{cases}
    \end{equation}
    The first equation in \eqref{equ: Ai = LiUs+LsUi} can be obtained from the LU decomposition of the standard part $A_{s}$.
    Treating $L_{i}$ and $U_{i}$ as unknowns, the second equation in \eqref{equ: Ai = LiUs+LsUi} can be seen as solving the generalized Sylvester equation as in Lemma \ref{lem: solution of Sylvester equation}.
    $L_{i}$ and $U_{i}$ have a general solution 
    \begin{equation}\label{equ: Li Ui by gSylvester}
        \begin{cases}
            L_{i} &= -(I_{m}-L_{s}L_{s}^{\dagger})A_{i}(-U_{s})^{\dagger}+Z-(I_{m}-L_{s}L_{s}^{\dagger})Z(-U_{s})(-U_{s})^{\dagger}\\
            U_{i} &= L_{s}^{\dagger}A_{i}+L_{s}^{\dagger}Z(-U_{s})+(I_{k}-L_{s}^{\dagger}L_{s})W,
        \end{cases}
    \end{equation}
    with $W\in\mathbb{R}^{k\times n}$ and $Z\in \mathbb{R}^{m\times k} $ being arbitrary iff
    \begin{equation}
        (I_{m}-L_{s}L_{s}^{\dagger})A_{i}(I_{n}-U_{s}^{\dagger}U_{s}) = O_{m\times n}.
    \end{equation}
    Considering that $L_{s}$ is a lower trapezoidal matrix of rank $k$, together with $L_{s}^{\dagger}L_{s} = I_{k}$, the second equation in \eqref{equ: Li Ui by gSylvester} becomes
    \begin{equation}
        U_{i} = L_{s}^{\dagger}A_{i}+L_{s}^{\dagger}Z(-U_{s})+(I_{k}-L_{s}^{\dagger}L_{s})W  = L_{s}^{\dagger}A_{i}-L_{s}^{\dagger}ZU_{s}.
    \end{equation}
    Similarly, $U_{s}$ is an upper trapezoidal matrix of rank $k$. 
    Combined with $U_{s}U_{s}^{\dagger} = I_{k}$, the first equation in \eqref{equ: Li Ui by gSylvester} becomes
    \begin{equation}
        L_{i}  =  -(I_{m}-L_{s}L_{s}^{\dagger})A_{i}(-U_{s})^{\dagger}+Z-(I_{m}-L_{s}L_{s}^{\dagger})Z(-U_{s})(-U_{s})^{\dagger} = (I_{m}-L_{s}L_{s}^{\dagger})A_{i}U_{s}^{\dagger}++L_{s}L_{s}^{\dagger}Z.
    \end{equation}
    To determine the exact expressions for $L_{i}$ and $U_{i}$, we need to determine the values of the $m\times k$ elements in $Z$.
    Recognizing that $L_{i}^{\dagger}$ is of size $k \times m$ and always appears before $Z$, let's define $P = L_{i}^{\dagger}Z\in \mathbb{R}^{k \times k}$. 
    This means we only need to determine the values of $k\times k (\leq m \times k)$ elements to determine the exact expressions for $L_{i}$ and $U_{i}$.
    The expressions for $L_{i}$ and $U_{i}$ are 
    \begin{equation}\label{equ: Li Ui in proof}
        \begin{cases}
            L_{i} & = (I_{m}-L_{s}L_{s}^{\dagger})A_{i}U_{s}^{\dagger} + L_{s}P, \\
            U_{i}& = L_{s}^{\dagger}A_{i}-PU_{s},
        \end{cases}
    \end{equation}
    with $P\in \mathbb{R}^{k \times k}$ being arbitrary. 
    It is worth noting that $U_{i}$ is an upper trapezoidal matrix. 
    For the second equation in \eqref{equ: Li Ui in proof}, when we set $B=L_{s}^{\dagger}A_{i} \in \mathbb{R}^{k \times n}$ and expand the matrix equation column-wise, we obtain the following result,
    \begin{equation}
        U_{i}(:,t) = B(:,t) - PU_{s}(:,t), \quad t = 1,2,\ldots,n.
    \end{equation}
    Since $U_{i}$ and $U_{s}$ are upper trapezoidal matrices with zeros in their lower triangle, we can uniquely determine the lower triangular part of $P$. 
    This can be obtained by solving the following equations,
    \begin{equation}\label{equ: matrix equation of lower P}
        \begin{cases}
            B(2:k,1) & = P(2:k,1)U_{s}(1,1),\\
            B(3:k,2) & = P(3:k,1)U_{s}(1,2)+P(3:k,2)U_{s}(2,2), \\
           & \ldots \\
           B(k,k-1) & = \sum_{l=1}^{k-1}P(k,l)U_{s}(l,k-1). 
        \end{cases}
    \end{equation}
    By solving \eqref{equ: matrix equation of lower P}, we obtain the expression for the lower triangular part of $P$ as follows,
    \begin{equation}
        \begin{cases}
            P(2:k,1) &= \frac{1}{U_{s}(1,1)}B(2:k,1),\\
            P(3:k,2) & = \frac{1}{U_{s}(2,2)}(B(3:k,2)-P(3:k,1)U_{s}(1,2)),\\
            &\ldots \\
            P(k,k-1) & = \frac{1}{U_{s}(k-1,k-1)}(B(k,k-1)-P(k,1:k-2)U_{s}(1:k-2,k-1)).
        \end{cases}
    \end{equation}
    Similarly, in the first equation of \eqref{equ: Li Ui in proof}, we set $C = (I_{m}-L_{s}L_{s}^{\dagger})A_{i}U_{s}^{\dagger} \in \mathbb{R}^{m\times k}$ and expand it row-wise to obtain,
    \begin{equation}
        L_{i}(t,:) = C(t,:)+L_{s}(t,:)P, \quad t = 1,2,\ldots, m
    \end{equation}
    Noting that $L_{i}$ and $L_{s}$ are lower trapezoidal matrices, by utilizing the upper triangular parts of $L_{s}$ and $L_{i}$ being 0, we can uniquely determine the upper triangular part of $P$. 
    This can be obtained by solving the following equations,
    \begin{equation}\label{equ: matrix equation of upper P}
        \begin{cases}
            -C(1,2:k) & = L_{s}(1,1)P(1,2:k),\\
            -C(2,3:k) & = L_{s}(2,1)P(1,3:k)+L_{s}(2,2)P(2,3:k),\\
            & \ldots \\
            -C(k-1,k) & = \sum_{l=1}^{k-1}L_{s}(k-1,l)P(l,k),
        \end{cases}
    \end{equation}
    By solving \eqref{equ: matrix equation of upper P}, we obtain the expression for the upper triangular part of $P$,
    \begin{equation}
        \begin{cases}
            P(1,2:k) &= \frac{1}{L_{s}(1,1)}(-C(1,2:k)), \\
            P(2,3:k) & =\frac{1}{L_{s}(2,2)}(-C(2,3:k)-L_{s}(2,1)P(1,3:k)),\\
            & \ldots \\
            P(k-1,k)&=\frac{1}{L_{s}(k-1,k-1)}(-C(k-1,k)-L_{s}(k-1,1:k-2)P(1:k-2,k)).
        \end{cases}
    \end{equation}
    We have now completed all the proofs in the theorem.
\end{proof}

\begin{remark}\label{rem: condition fail}
    When designing the dual matrix LU algorithm, if \eqref{equ: (I-LsLsT)Ai(I-UsUs)} is not satisfied, can we still obtain other approximation results? The answer is yes. Viewing the LU decomposition of the dual matrix $A=A_{s} + A_{i}\epsilon$ as a rank-$k$ approximation of $A$, we refer to the analysis provided by Wei et al. \cite{wei2024singular} for rank-$k$ approximations of dual matrices. In the metric induced by the Frobenius norm, as long as $A$'s standard part $A_{s}$ is the optimal rank-$k$ approximation $A_{s}^{(k)} = L_{s}^{(k)}U_{s}^{(k)}$, the obtained decomposition is always the optimal rank-$k$ approximation. In the metric of the quasi-norm (if $A_{s}\neq B_{s}$, $d_{*}(A,B) =\|A_{s}-B_{s}\|_{F}+\frac{\|A_{i}-B_{i}\|_{F}^{2}}{2\|A_{s}-B_{s}\|_{F}}\epsilon$; otherwise,  $d_{*}(A,B)=\|A_{i}-B_{i}\|_{F}\epsilon$.), when $A$'s standard part $A_{s}^{(k)} = L_{s}^{(k)}U_{s}^{(k)}$ is the optimal rank-$k$ approximation, then by defining $A_{i}^{(k)}= A_{i}-(I_{m}-L_{s}^{(k)}L_{s}^{(k)\dagger})A_{i}(I_{n}-U_{s}^{(k)\dagger}U_{s}^{(k)})$, the new $A_{i}^{(k)}$ satisfies condition \eqref{equ: (I-LsLsT)Ai(I-UsUs)}, and using the LU decomposition calculated according to \cref{the: DLU decomposition} provides the optimal rank-$k$ approximation for the dual matrix $A$.
\end{remark}

\begin{remark}\label{rem: DLU for C&H}
    According to Theorem \ref{the: DLU decomposition}, the LU decomposition of dual matrices over the real numbers holds under the condition that the LU decomposition over the real numbers exists and the generalized Sylvester equation in Lemma \ref{lem: solution of Sylvester equation} has a solution. 
    When considering the complex and quaternion cases, the LU decomposition exists for both complex and quaternion matrices \cite{zhang1997quaternions}. 
    However, while Lemma \ref{lem: solution of Sylvester equation} holds over the complex field \cite{Baksalary1979TheME}, its validity over the quaternions has not yet been studied.
    Consequently, a corresponding LU decomposition for dual complex matrices exists, whereas the existence of an LU decomposition for dual quaternion matrices remains unknown.
\end{remark}

Based on the analysis of Theorem \ref{the: DLU decomposition} and  Remark \ref{rem: condition fail}, we provide the pseudocode for the LU decomposition of the dual matrix in Algorithm \ref{alg: LU Decomposition of Dual Matrix}.

\begin{algorithm}[htb]
\DontPrintSemicolon
    \KwInput{$A = A_{s}+A_{i}\epsilon \in \mathbb{DR}^{m \times n}$, integer $k$.}
    
     Compute the rank-$k$ LU decomposition of $A_{s} = L_{s}U_{s}$ .

    \If{$(I_{m}-L_{s}L_{s}^{\dagger})A_{i}(I_{n}-U_{s}^{\dagger}U_{s}) \neq O_{m\times n}$}{

    $A_{i} = A_{i} - (I_{m}-L_{s}L_{s}^{\dagger})A_{i}(I_{n}-U_{s}^{\dagger}U_{s})$.}
    
    $B = L_{s}^{\dagger}A_{i}, \quad C = (I_{m}-L_{s}L_{s}^{\dagger})A_{i}U_{s}^{\dagger}$. 

    $P = zeros(k,k)$.

    $P(2:k,1) = B(2:k,1)/U_{s}(1,1), \quad P(1,2:k) = -C(1,2:k)/L_{s}(1,1)$.
    
    \For{$i = 2:k-1$}{

        $P(i+1:k,i)=(B(i+1:k,i)-P(i+1:k,1:i-1)U_{s}(1:i-1,i))/U_{s}(i,i)$.

        $P(i,i+1:k)=(-C(i,i+1:k)-L_{s}(i,1:i-1,1)P(1:i-1,i+1:k))/L_{s}(i,i)$.
    }

    $L_{i} = C+L_{s}P$.

    $U_{i} = B-PU_{s}$.

    \KwOutput{$L = L_{s}+L_{i}\epsilon \in \mathbb{DR}^{m \times k}$ is a lower trapezoidal dual matrix, and $U = U_{s}+U_{i}\epsilon \in \mathbb{DR}^{k \times n}$ is an upper trapezoidal dual matrix.}
    \caption{LU Decomposition of the Dual Matrix (DLU)}
    \label{alg: LU Decomposition of Dual Matrix}
\end{algorithm}

\subsection{DMPGI and Rank of Dual Matrices}\label{sec: DMPGI and Rank of Dual Matrices}
It is evident that there exists an LU decomposition $A = LU$, where $L\in \mathbb{R}^{m \times k}$ and $U\in \mathbb{R}^{k \times n}$, for a given rank-$k$ matrix $A\in \mathbb{R}^{m \times n}$. 
The question arises as to whether a similar conclusion can be drawn for dual matrices.
Wang, Cui, and Liu \cite{wang2023dual} introduced the concept of a full-rank dual matrix, where the standard part is a matrix of full rank. 
They subsequently devised the rank-$k$ decomposition for such dual matrices.
The rank-$k$ decomposition $A=BC$ of a dual matrix $A$ exists iff the rank-$k$ decomposition $A_{s}=B_{s}C_{s}$ of its standard part $A_{s}$ exists and satisfies $(I_{m}-B_{s}B_{s}^{\dagger})A_{i}(I_{n}-C_{s}^{\dagger}C_{s}) = O_{m\times n}$.
The following lemma gives the conditions for the existence of DMPGI for a rank-$k$ dual matrix.

\begin{lemma}\label{lem: MPDGI}
    \cite{wang2021characterizations} Let $A = A_{s} + \epsilon A_{i} \in\mathbb{DR}^{m \times n}$, and $\rank(A_{s}) = k $. 
    Therefore, the existence of the DMPGI for matrix $A$ is indicated if and only if
    \begin{equation}
        (I_{m}-A_{s}A_{s}^{\dagger})A_{i}(I_{n}-A_{s}^{\dagger}A_{s}) = O_{m \times n}.
    \end{equation}
    Moreover,
    \begin{equation}
        A^{\dagger} = A_{s}^{\dagger} - \epsilon(A_{s}^{\dagger}A_{i}A_{s}^{\dagger}-(A_{s}^{\top}A_{s})^{\dagger}A_{i}^{\top}(I_{m}-A_{s}A_{s}^{\dagger})-(I_{n}-A_{s}^{\dagger}A_{s})A_{i}^{\top}(A_{s}A_{s}^{\top})^{\dagger}).
    \end{equation}
\end{lemma}

Motivated by the rank-$k$ decomposition of dual matrices and considering Lemma \ref{lem: MPDGI}, we conjecture that the existence of a rank-$k$ decomposition, the existence of DLU decomposition, and the existence of DMPGI are interrelated. 
The theorem providing the proof is presented below.

\begin{theorem}
Consider a matrix $A = A_{s}+\epsilon A_{i}\in \mathbb{DR}^{m \times n} (m \geq n)$ with $\rank(A_{s}) = k \leq \min\{m,n\}$. 
Let $A_{s}=L_{s}U_{s}$ represent the LU decomposition of $A_{s}$, where $L_{s} \in \mathbb{R}^{m \times k}$ is a lower trapezoidal matrix with rank $k$, and $U_{s} \in \mathbb{R}^{k \times n}$ is an upper trapezoidal matrix with rank $k$.
The following conditions are equivalent,
    \begin{enumerate}[label = (\arabic*)]
        \item the DMPGI of $A$ exists;
        \item the dual rank-$k$ decomposition of $A$ exists;
        \item the DLU decomposition of $A$ exists; 
        \item $(I_{m}-L_{s}L_{s}^{\dagger})A_{i}(I_{n}-U_{s}^{\dagger}U_{s}) = O_{m\times n}$.
    \end{enumerate}
\end{theorem}
\begin{proof}
    $(4)\Rightarrow (3)$: By Theorem \ref{the: DLU decomposition}, this is apparent.
    
    $(3)\Rightarrow (2)$: Assume that the DLU decomposition of $A$ exists, then there exists $A=LU$, where $L=L_{s}+L_{s}\in \mathbb{DR}^{m\times k}$ and $U=U_{s}+U_{i}\in \mathbb{DR}^{k \times n}$.
    The DLU decomposition of $A$ is the dual rank-$k$ decomposition of $A$.

    $(2) \Rightarrow (1)$: Assuming the existence of the dual rank-$k$ decomposition of $A$, denoted as $A=BC$, where $A_{s}=B_{s}C_{s}$, with $(I_{m}-B_{s}B_{s}^{\dagger})A_{i}(I_{n}-C_{s}^{\dagger}C_{s}) = O_{m\times n}$.
    Considering that $\rank(A_{s}) = k$, then $A_{s}=B_{s}C_{s}$ is a full rank decomposition. $B_{s}$ has full column rank, and $C_{s}$ has full row rank. 
    This implies $B_{s}B_{s}^{\dagger}=A_{s}A_{s}^{\dagger}$ and $C_{s}^{\dagger}C_{s} = A_{s}^{\dagger}A_{s}$.
    So we have $(I_{m}-A_{s}A_{s}^{\dagger})A_{i}(I_{n}-A_{s}^{\dagger}A_{s}) = O_{m \times n}$, and combined with Lemma \ref{lem: MPDGI}, then the DMPGI of $A$ exists.

    $(1) \Rightarrow  (4)$: According to Lemma \ref{lem: MPDGI}, the DMPGI of $A$ exists, so we have $(I_{m}-A_{s}A_{s}^{\dagger})A_{i}(I_{n}-A_{s}^{\dagger}A_{s}) = O_{m \times n}$.
    Note that $\rank(A_{s}) = k$, so $A_{s}=L_{s}U_{s}$ is a full rank decomposition. 
    $L_{s}$ has full column rank, and $U_{s}$ has full row rank. 
    This implies $A_{s}A_{s}^{\dagger}=L_{s}L_{s}^{\dagger}$ and $A_{s}^{\dagger}A_{s} = U_{s}^{\dagger}U_{s}$.
    So $O_{m \times n} = (I_{m}-A_{s}A_{s}^{\dagger})A_{i}(I_{n}-A_{s}^{\dagger}A_{s}) = (I_{m}-L_{s}L_{s}^{\dagger})A_{i}(I_{n}-U_{s}^{\dagger}U_{s})$.

    Consequently, we conclude that conditions (1), (2), (3), and (4) are equivalent.
    
\end{proof}
{\color{red}
Cui and Qi \cite{cui2025genuine} proposed a new dual Moore-Penrose inverse (NDMPI) and further explored its relationship and differences with the Moore-Penrose dual generalized inverse (MPDGI) and DMPGI.
}
\subsection{Dual Cholesky Decomposition}\label{sec: ESDCD}
In the solution of large-scale systems of symmetric positive definite linear equations, Cholesky decomposition holds significant importance. 
Subsequently, we present an explicit expression for the dual matrix, preceded by the introduction of a lemma.
\begin{lemma}\label{lem: A'X+X'A=B}
    \cite{djordjevic2007explicit}  Given $A\in \mathbb{R}^{n\times m} $ is invertible, and $B \in \mathbb{R}^{m \times m}$. The necessary and sufficient condition for the existence of a solution to the equation
    \begin{equation}\label{equ: A'X+X'A=B}
        A^{\top}X +X^{\top}A = B.
    \end{equation}
    is $B^{\top}=B$, and there exists a matrix $Z\in \mathbb{R}^{n \times n}$ satisfying $Z^{\top} = -Z$, the general solution expression to \eqref{equ: A'X+X'A=B} is given by
    \begin{equation}
        X = \frac{1}{2}(A^{\top})^{-1}B+ZA.
    \end{equation}
\end{lemma}

Consider a symmetric positive definite dual matrix $A=A_{s}+\epsilon A_{i}\in \mathbb{DR}^{n \times n}$. 
Assuming $A$ has a dual Cholesky decomposition $A=LL^{\top}$, where $L = L_{s}+\epsilon L_{i} \in \mathbb{DR}^{n\times n}$ exists. 
This leads to two matrix equations,
\begin{equation}\label{equ: Ai = LsLi+LiLs}
    \begin{cases}
        A_{s} & = L_{s}L_{s}^{\top}, \\
        A_{i} & = L_{s}L_{i}^{\top} + L_{i}L_{s}^{\top}.
    \end{cases}
\end{equation}
The first equation is the standard Cholesky decomposition, while the second equation can be derived using Lemma \ref{lem: A'X+X'A=B}.
The idea of solving these two problems together allows us to derive the explicit solution of dual Cholesky decomposition.

\begin{theorem}\label{the: dual cholesky decomposition}
    (Dual Cholesky Decomposition) Consider a dual symmetric positive definite matrix $A = A_{s}+A_{i} \epsilon \in \mathbb{DR}^{n \times n}$.
    Suppose $A_{s}=L_{s}L_{s}^{\top}$ represents the Cholesky decomposition of $A_{s}$, and $A_{s}$ is assumed to be invertible.
    Then the Cholesky decomposition of the dual matrix $A = LL^{\top}$, where $L = L_{s}+\epsilon L_{i} \in \mathbb{DR}^{n \times n}$ is a lower triangular dual matrix exist. Additionally, the expressions of $L_{i}$, the infinitesimal part of $L$ is 
    \begin{equation}\label{equ: Li expression in dual cholesky decomposition}
        L_{i}=\frac{1}{2}A_{i}(L_{s}^{\top})^{-1}-L_{s}P,
    \end{equation}
    where $P\in \mathbb{R}^{n \times n}$ is a skew-symmetric matrix. The expression of the upper triangular part of $P$ as
    \begin{equation}
    P(1:t-1,t) = L_{s}^{-1}(1:t-1,1:t-1)B(1:t-1,t),\quad t=2,3,\ldots,n,
    \end{equation}
    where $B =\frac{1}{2}A_{i}(L_{s}^{\top})^{-1} \in \mathbb{DR}^{n\times n}$.
\end{theorem}
\begin{proof}
    Assume that the Cholesky decomposition of the dual symmetric positive definite matrix $A$ exists, the standard and infinitesimal parts of $A$ can be expressed as
    \begin{equation}\label{equ: Ai = LiLs+LsLi}
        \begin{cases}
            A_{s} & = L_{s}L_{s}^{\top}, \\
            A_{i} & = L_{s}L_{i}^{\top} + L_{i}L_{s}^{\top}.
        \end{cases}
    \end{equation}
    The first equation in \eqref{equ: Ai = LiLs+LsLi} can be obtained from the Cholesky decomposition of the standard part $A_{s}$.
    Treating $L_{i}$ as unknown and $A_{s}$ is invertible, the second equation in \eqref{equ: Ai = LiLs+LsLi} can be seen as solving the equation as in Lemma \ref{lem: A'X+X'A=B}.
    $L_{i}$ has a solution 
    \begin{equation}\label{equ: Li by Lemma}
         L_{i}=\frac{1}{2}A_{i}(L_{s}^{\top})^{-1}-L_{s}P,
    \end{equation}
    where $P\in \mathbb{R}^{n \times n}$ is an arbitrary skew-symmetric matrix. 
    Just as in Theorem \ref{the: DLU decomposition}, we now determine the elements in $P$ through the lower triangular property of $L_{i}$. 
    According to \eqref{equ: Li by Lemma} and set $B =\frac{1}{2}A_{i}(L_{s}^{\top})^{-1}$, we obtain the matrix equation,
    \begin{equation}\label{equ: Li = B +PLs}
        O = B(1:t-1,t)-L_{s}(1:t-1,1:t-1)P(1:t-1,t),\quad t =2,3,\cdots,n.
    \end{equation}
    Solving this set of equations, we can obtain the expression of the upper triangular part of $P$ by
    \begin{equation}
        P(1:t-1,t) = L_{s}^{-1}(1:t-1,1:t-1)B(1:t-1,t),\quad t=2,3,\ldots,n.
    \end{equation}
    The remaining elements in $P$ can be obtained using the skew-symmetry of $P$.     
\end{proof}

\begin{remark}\label{rem: semidefine}
    For the dual Cholesky decomposition in Theorem \ref{the: dual cholesky decomposition}, we consider the case where the dual matrix $A=A_{s}+\epsilon A_{i} \in \mathbb{R}^{n\times n}$ is symmetric and positive semi-definite. 
    The questions of existence and uniqueness of a Cholesky decomposition when $A_{s} = L_{s}L_{s}^{\top}$ is positive semi-definite are answered by the following result \cite{higham1990analysis}.
    The construction process for the infinitesimal part $L_{i}$ is analogous to the proof in Theorem \ref{the: dual cholesky decomposition}. 
    Therefore, the dual Cholesky decomposition of symmetric positive semi-definite dual matrices also exists.
\end{remark}

Based on Theorem \ref{the: dual cholesky decomposition}, we can obtain pseudocode for the dual Cholesky decomposition (DCholesky) in Algorithm \ref{alg: Cholesky Decomposition of Dual Matrix}.

\begin{algorithm}[htb]
\DontPrintSemicolon
    \KwInput{daul symmetric positve definite matrix $A = A_{s}+A_{i}\epsilon \in \mathbb{DR}^{n \times n}$ and $A_{s}$ is invertible.}
    
     Compute the Cholesky decomposition of $A_{s} = L_{s}L_{s}^{\top}$ .

    $B =\frac{1}{2}A_{i}(L_{s}^{\top})^{-1}$. 

    $P = zeros(n,n)$.

    
    \For{$t = n:-1:2$}{

        $P(1:t-1,t)  = L_{s}^{-1}(1:t-1,1:t-1)B(1:t-1,t)$.

        $P(t,1:t-1)=-P(1:t-1,t)^{\top}$.
    }

    $L_{i}=B-L_{s}P $.

    \KwOutput{$L = L_{s}+L_{i}\epsilon \in \mathbb{DR}^{n \times n}$ is a lower trangular dual matrix.}
    \caption{Dual Cholesky decomposition (DCholesky)}
    \label{alg: Cholesky Decomposition of Dual Matrix}
\end{algorithm}

\section{Experiments}\label{sec: experiment}

In this section, we use a random dual real symmetric matrix as an example to calculate ATDSVD and DLU decomposition. 
We use the Hankel structured noise of the real symmetric Hankel matrix to obtain a detailed error decomposition.
All experiments were conducted on an Apple M2 Pro processor with 16GB of memory, using MATLAB version 2023b with a machine precision of $2.22e-16$.

\begin{example}\label{exam: random A}
    Construct a random dual real symmetric matrix $A$, where
\begin{equation}
    A_{s} = \begin{bmatrix}
         0.4910  &  0.4263  &  0.3317 &   0.8574\\
    0.4263  &   1.5287    & 1.1186  &   1.7450\\
    0.3317   &  1.1186   &  1.0930  &   1.3879\\
    0.8574   &  1.7450 &    1.3879  &   2.4994
    \end{bmatrix}, A_{i} = \begin{bmatrix}
        1.1980  &  0.9304  &  1.0057  &  0.9948 \\
    0.9304  &  0.8665  &  1.1222  &  0.8256\\
    1.0057  &  1.1222  &  2.0469  &  1.1378\\
    0.9948  &  0.8256  &  1.1378  &  1.0240
    \end{bmatrix}.
\end{equation}
Using ATDSVD from Algorithm \ref{alg: ATDCD}, we can obtain the SVD of $A = VSV^{\top}$, where
\begin{equation}
        V_{s} = \begin{bmatrix}
         -0.2182  &  0.7044  & -0.1299 &  -0.6628 \\
   -0.5280 &  -0.4138   & 0.6310  & -0.3896 \\
   -0.4268  & -0.4469  & -0.7642  & -0.1847 \\
   -0.7010  &  0.3645  &  0.0304  &  0.6122
    \end{bmatrix}, 
         V_{i} = \begin{bmatrix}
       -0.2237  &  0.0523   &-1.8169  &  0.4854\\
    0.0509  &  0.5157 &  -0.3976 &  -1.2605\\
   -0.1896  & -0.5158 &   0.0263 &   1.5775\\
    0.1468  & -0.1482 &   1.1493  &  0.1992\\
    \end{bmatrix},
\end{equation}
and 
\begin{equation}
\begin{cases}
    S_{s} = \diag\{ 4.9258 ,0.4738,0.1705,0.0421\},\\
    S_{i} = \diag\{3.6787   , 0.4183 , 0.4973 ,0.5411\}.
\end{cases}   
\end{equation}
Using Algorithm \ref{alg: Cholesky Decomposition of Dual Matrix}, we can obtain the DCholesky decomposition of $A$, where
\begin{equation}
    L_{s} =\begin{bmatrix}
         0.7007   &      0      &   0    &     0\\
    0.6084  &  1.0763   &      0  &       0 \\
    0.4733   & 0.7717  &  0.5229   &      0 \\
    1.2235   & 0.9296 &   0.1748   & 0.3281
    \end{bmatrix},
    L_{i} = \begin{bmatrix}
        0.8548 &  0       & 0  &       0 \\
    0.5856  &  0.0715    &     0 &        0\\
   0.8578  & 0.2489   & 0.8133  &       0\\
   -0.0731  & 0.0808  &  -0.5987   & 1.9236
    \end{bmatrix}.
\end{equation}
\end{example}

\begin{example}\label{exam: Hankel}
    Hankel matrices are a special class of symmetric matrices and a special type of Toeplitz matrices.
    Hankel matrices have important applications in various fields, such as signal processing \cite{zhao2009similarity} and control theory of continuous or discrete systems \cite{mishra2020data}.
    Consider a real Hankel matrix $H_{s}\in \mathbb{R}^{5\times 5}$.
    The noise of $H_{s}$ is also a real Hankel matrix, denoted by $H_{i}\mathbb{R}^{5\times 5}$, where
    \begin{equation}
        H_{s} = \begin{bmatrix}
        1 &  0  &  1  &  0  & 1\\
        0 &  1  &  0  &  1  & 0\\
        1 &  0  &  1  &  0  & 1\\
        0 &  1  &  0  &  1  & 0\\
        1 &  0  &  1  &  0  & 1\\
    \end{bmatrix}, H_{i}= \begin{bmatrix}
        1 &  0  &  1  &  0  & 1\\
        0 &  1  &  0  &  1  & 0\\
        1 &  0  &  1  &  0  & 1\\
        0 &  1  &  0  &  1  & 0\\
        1 &  0  &  1  &  0  & 1\\
    \end{bmatrix}.
    \end{equation}
    The dual Hankel matrix $H =H_{s}+\epsilon H_{i}$. 
    By calculating ATDSVD from Algorithm \ref{alg: ATDCD}, we obtain $H=U\Sigma U^{\top}$, where
    \begin{equation}
    \begin{cases}
        U_{s} = \begin{bmatrix}
             -0.5774  &  0&    0&    0.8165 &  0 \\
         0   & 0.7071 &   0.3801&         0  & -0.5963 \\
   -0.5774   &0 & -0.5963  & -0.4082  & -0.3801 \\
         0   & 0.7071 &-0.3801    &     0  &  0.5963 \\
   -0.5774   &0 &0.5963   &-0.4082  &  0.3801
        \end{bmatrix}, \\
        U_{i} = 10^{-15}*\begin{bmatrix}
            0.1209  & 0  & -0.1470  & -0.0855  &  0 \\
    0.0281&   -0.1506   & 0.4065&    0.0883   & 0.0805\\
    0.0163 &   0.0902  &  0.2966 &  -0.2240  & -0.2494\\
   -0.0281  &  0.1506 &  -0.0924  & -0.0883 &  -0.2375\\
    0.1046   &-0.0902&   -0.0214   & 0.0531&    0.2494\\
        \end{bmatrix},
    \end{cases}         
    \end{equation}
    and
    \begin{equation}
        \begin{cases}
            \Sigma_{s} =\diag\{3,2,0,0,0\}, \\
            \Sigma_{i} =\diag\{3,2,0,0,0\}. 
        \end{cases}
    \end{equation}
    ATDSVD provides a detailed decomposition of the Hankel signal noise, which is of great help for the analysis of Hankel matrices. 
    It should be noted that $H_{s}$ is not a symmetric positive definite matrix, so $H$ does not have a DLU decomposition.
\end{example}

\section{Conclusion}
This article utilizes matrix transformations to present the Autoone-Takagi decomposition of dual complex symmetric matrices.
It extends to the generalized Autoone-Takagi decomposition of dual quaternion $\eta$-Hermitian matrices. 
Simultaneously, the LU decomposition of dual matrices can be obtained using the general solution of the Sylvester equation.
The existence of LU decomposition is equivalent to the existence of rank-$k$ decomposition and DMPGI.
The Cholesky decomposition of dual matrices is provided in the case of real symmetric matrices. 
These theoretical explorations, driven by numerical linear algebra, lay foundational contributions to the theory of dual matrices.

\section{Data availability}
No data was used for the research described in the article.
\section{Acknowledgements}
Renjie Xu is supported by the Hong Kong Innovation and Technology Commission (InnoHK Project CIMDA).
Yimin Wei is supported by the National Natural Science Foundation of China under grant 12271108, and the Ministry of Science and Technology of China under grant G2023132005L.
Hong Yan is supported by the Hong Kong Research Grants Council (Project 11204821), the Hong Kong Innovation and Technology Commission (InnoHK Project CIMDA), and the City University of Hong Kong (Projects 9610034 and 9610460).

\footnotesize
\bibliographystyle{abbrv}
\bibliography{reference}

\begin{thebibliography}{10}

\bibitem{autonne1915matrices}
L.~Autonne.
\newblock {\em Sur les matrices hypohermitiennes et sur les matrices unitaires}, volume~38.
\newblock A. Rey, 1915.

\bibitem{Baksalary1979TheME}
J.~K. Baksalary and R.~Kala.
\newblock The {M}atrix {E}quation ${AX - YB = C}$.
\newblock {\em Linear Algebra and its Applications}, 25:41--43, 1979.

\bibitem{baydin2018automatic}
A.~G. Baydin, B.~A. Pearlmutter, A.~A. Radul, and J.~M. Siskind.
\newblock Automatic differentiation in machine learning: a survey.
\newblock {\em Journal of Marchine Learning Research}, 18:1--43, 2018.

\bibitem{che2018adaptive}
M.~Che, S.~Qiao, and Y.~Wei.
\newblock Adaptive algorithms for computing the principal takagi vector of a complex symmetric matrix.
\newblock {\em Neurocomputing}, 317:79--87, 2018.

\bibitem{cui2025genuine}
C.~Cui and L.~Qi.
\newblock A genuine extension of the {M}oore--{P}enrose inverse to dual matrices.
\newblock {\em Journal of Computational and Applied Mathematics}, 454:116185, 2025.

\bibitem{djordjevic2007explicit}
D.~S. Djordjevi{\'c}.
\newblock Explicit solution of the operator equation a* x+ x* a= b.
\newblock {\em Journal of Computational and Applied Mathematics}, 200(2):701--704, 2007.

\bibitem{GU1987dual}
Y.-L. Gu and J.~Luh.
\newblock Dual-number transformation and its applications to robotics.
\newblock {\em IEEE Journal on Robotics and Automation}, 3(6):615--623, 1987.

\bibitem{gutin2022generalizations}
R.~Gutin.
\newblock Generalizations of singular value decomposition to dual-numbered matrices.
\newblock {\em Linear and Multilinear Algebra}, 70(20):5107--5114, 2022.

\bibitem{higham1990analysis}
N.~J. Higham.
\newblock {Analysis of the Cholesky decomposition of a semi-definite matrix}.
\newblock In {\em {Reliable Numerical Commputation}}. Oxford University Press, 09 1990.

\bibitem{horn2012matrix}
R.~A. Horn and C.~R. Johnson.
\newblock {\em Matrix analysis}.
\newblock Cambridge university press, 2012.

\bibitem{horn2012generalization}
R.~A. Horn and F.~Zhang.
\newblock A generalization of the complex autonne--takagi factorization to quaternion matrices.
\newblock {\em Linear and Multilinear Algebra}, 60(11-12):1239--1244, 2012.

\bibitem{kavan2006dual}
L.~Kavan, S.~Collins, C.~O’Sullivan, and J.~Zara.
\newblock Dual quaternions for rigid transformation blending.
\newblock {\em Trinity College Dublin}, 5, 2006.

\bibitem{liu2024newsvendor}
C.~Liu and W.~Zhu.
\newblock Newsvendor conditional value-at-risk minimisation: A feature-based approach under adaptive data selection.
\newblock {\em European Journal of Operational Research}, 313(2):548--564, 2024.

\bibitem{luo2022evaluating}
Y.-K. Luo, S.-X. Chen, L.~Zhou, and Y.-Q. Ni.
\newblock Evaluating railway noise sources using distributed microphone array and graph neural networks.
\newblock {\em Transportation Research Part D: Transport and Environment}, 107:103315, 2022.

\bibitem{mishra2020data}
V.~K. Mishra, I.~Markovsky, and B.~Grossmann.
\newblock Data-driven tests for controllability.
\newblock {\em IEEE Control Systems Letters}, 5(2):517--522, 2020.

\bibitem{pennestri2009linear}
E.~Pennestr{\`\i} and P.~P. Valentini.
\newblock Linear dual algebra algorithms and their application to kinematics.
\newblock In {\em Multibody Dynamics: Computational Methods and Applications}, pages 207--229. Springer, 2009.

\bibitem{perez2004dual}
A.~Perez and J.~M. McCarthy.
\newblock Dual quaternion synthesis of constrained robotic systems.
\newblock {\em J. Mech. Des.}, 126(3):425--435, 2004.

\bibitem{qi2023eigenvalues}
L.~Qi and C.~Cui.
\newblock Eigenvalues and jordan forms of dual complex matrices.
\newblock {\em Communications on Applied Mathematics and Computation}, pages 1--17, 2023.

\bibitem{qi2022dual}
L.~Qi, C.~Ling, and H.~Yan.
\newblock Dual quaternions and dual quaternion vectors.
\newblock {\em Communications on Applied Mathematics and Computation}, 4(4):1494--1508, 2022.

\bibitem{shabat2018randomized}
G.~Shabat, Y.~Shmueli, Y.~Aizenbud, and A.~Averbuch.
\newblock Randomized lu decomposition.
\newblock {\em Applied and Computational Harmonic Analysis}, 44(2):246--272, 2018.

\bibitem{takagi1924algebraic}
T.~Takagi.
\newblock On an algebraic problem reluted to an analytic theorem of carath{\'e}odory and fej{\'e}r and on an allied theorem of landau.
\newblock In {\em Japanese journal of mathematics: transactions and abstracts}, volume~1, pages 83--93. The Mathematical Society of Japan, 1924.

\bibitem{teretenkov2022singular}
A.~Teretenkov.
\newblock Singular value decomposition for skew-takagi factorization with quantum applications.
\newblock {\em Linear and Multilinear Algebra}, 70(22):7762--7769, 2022.

\bibitem{toledo1997locality}
S.~Toledo.
\newblock Locality of reference in {LU} decomposition with partial pivoting.
\newblock {\em SIAM Journal on Matrix Analysis and Applications}, 18(4):1065--1081, 1997.

\bibitem{wang2023sar}
C.~Wang, S.~Luo, Y.~Huang, J.~Pei, Y.~Zhang, and J.~Yang.
\newblock Sar atr method with limited training data via an embedded feature augmenter and dynamic hierarchical-feature refiner.
\newblock {\em IEEE Transactions on Geoscience and Remote Sensing}, 2023.

\bibitem{wang2023crucial}
C.~Wang, S.~Luo, J.~Pei, Y.~Huang, Y.~Zhang, and J.~Yang.
\newblock Crucial feature capture and discrimination for limited training data sar atr.
\newblock {\em ISPRS Journal of Photogrammetry and Remote Sensing}, 204:291--305, 2023.

\bibitem{wang2021characterizations}
H.~Wang.
\newblock Characterizations and properties of the {MPDGI} and {DMPGI}.
\newblock {\em Mechanism and Machine Theory}, 158:104212, 2021.

\bibitem{wang2023dual}
H.~Wang, C.~Cui, and X.~Liu.
\newblock Dual r-rank decomposition and its applications.
\newblock {\em Computational and Applied Mathematics}, 42(8):349, 2023.

\bibitem{wang2024algebraic}
T.~Wang, Y.~Li, M.~Wei, Y.~Xi, and M.~Zhang.
\newblock Algebraic method for {LU} decomposition of dual quaternion matrix and its corresponding structure-preserving algorithm.
\newblock {\em Numerical Algorithms}, pages 1--16, 2024.

\bibitem{wang2017complex}
X.~Wang, M.~Che, and Y.~Wei.
\newblock Complex-valued neural networks for the takagi vector of complex symmetric matrices.
\newblock {\em Neurocomputing}, 223:77--85, 2017.

\bibitem{wei2024singular}
T.~Wei, W.~Ding, and Y.~Wei.
\newblock Singular value decomposition of dual matrices and its application to traveling wave identification in the brain.
\newblock {\em SIAM Journal on Matrix Analysis and Applications}, 45(1):634--660, 2024.

\bibitem{xu2024cur}
R.~Xu, S.~Feng, Y.~Wei, and H.~Yan.
\newblock {CUR} and generalized {CUR} decompositions of quaternion matrices and their applications.
\newblock {\em Numerical Functional Analysis and Optimization}, 45(3):234--258, 2024.

\bibitem{xu2024qr}
R.~Xu, T.~Wei, Y.~Wei, and P.~Xie.
\newblock {QR} decomposition of dual matrices and its application.
\newblock {\em Applied Mathematics Letters}, 156:109144, 2024.

\bibitem{xu2024utv}
R.~Xu, T.~Wei, Y.~Wei, and H.~Yan.
\newblock {UTV} decomposition of dual matrices and its applications.
\newblock {\em Computational and Applied Mathematics}, 43(1):41, 2024.

\bibitem{xu2024randomized}
R.~Xu and Y.~Wei.
\newblock Randomized quaternion matrix {UTV} decomposition and its applications in quaternion matrix optimization.
\newblock {\em Pacific Journal of Optimization}, 20(2):185--211, 2024.

\bibitem{zhang1997quaternions}
F.~Zhang.
\newblock Quaternions and matrices of quaternions.
\newblock {\em Linear algebra and its applications}, 251:21--57, 1997.

\bibitem{zhao2009similarity}
X.~Zhao and B.~Ye.
\newblock Similarity of signal processing effect between hankel matrix-based svd and wavelet transform and its mechanism analysis.
\newblock {\em Mechanical Systems and Signal Processing}, 23(4):1062--1075, 2009.

\bibitem{zhu2022quartic}
W.~Zhu and C.~Cartis.
\newblock Quartic polynomial sub-problem solutions in tensor methods for nonconvex optimization.
\newblock In {\em NeurIPS 2022 Workshop}, 2022.

\end{thebibliography}

\end{document}